\documentclass[11pt]{amsart}
\usepackage{amsfonts,amsmath,latexsym,color,epsfig,hyperref,mathdots,tikz,centernot}
\newtheorem{theorem}{Theorem}

\def\qed{\ifvmode\mbox{ }\else\unskip\fi\hskip 1em plus 10fill$\Box$}

\input{epsf}
\makeatletter
\def\Ddots{\mathinner{\mkern1mu\raise\p@
		\vbox{\kern7\p@\hbox{.}}\mkern2mu
		\raise4\p@\hbox{.}\mkern2mu\raise7\p@\hbox{.}\mkern1mu}}
\makeatother

\author{Quentin Dubroff}
\address{Department of Mathematics, Rutgers University, Piscataway, NJ, USA}
\email{quentin.dubroff@rutgers.edu}
\author{Jacob Fox}
\thanks{Fox is supported by a Packard Fellowship and by NSF award DMS-1855635. Xu is supported by the Cuthbert C. Hurd Graduate Fellowship in the Mathematical Sciences, Stanford}
\address{Department of Mathematics, Stanford University, Stanford, CA, USA}
\email{jacobfox@stanford.edu}
\author{Max Wenqiang Xu}
\address{Department of Mathematics, Stanford University, Stanford, CA, USA}
\email{maxxu@stanford.edu}

\begin{document}
	\title{\vspace{-0.7cm} A note on the Erd\H{o}s distinct subset sums problem}
	\maketitle
	
	\begin{abstract}
		We present two short proofs giving the best known asymptotic lower bound for the maximum element in a set of $n$ positive integers with distinct subset sums. 
	\end{abstract}
	
	Let $\{a_1,\ldots,a_n\}$ be a set of positive integers with $a_1<\ldots<a_n$ and all subset sums distinct.  
	Erd\H{o}s conjectured that $a_n \ge c \cdot 2^{n}$ for some constant $c>0$ and offered $\$ 500$ for a proof or disproof. See \cite{Conjecture} for more history on this. Using the second moment method, Erd\H{o}s and Moser \cite{Erdos Moser} (see also   \cite{AlSp}) proved 
	\[a_n \ge \frac{1}{4}\cdot n^{-1/2}\cdot 2^{n}.\]
	There have been some improvements on the constant factor of the above lower bound, including the work of Guy \cite{Guy}, Elkies \cite{Elkies}, Bae \cite{Bae}, and Aliev \cite{Aliev}. In particular, the previous best published lower bound was due to Aliev, which gave the constant factor $\sqrt{3/2\pi}$. In the other direction, the best known construction is due to Bohman \cite{Bohman bound}, where he constructed such sets with $a_n \le 0.22002 \cdot 2^{n}$.

	In this note, we  give two proofs of the best known lower bound, matching an unpublished result of
	Elkies and Gleason (see \cite{Aliev}).
	\begin{theorem}\label{main}
		If a set $\{a_1,\ldots,a_n\}$ of integers with $0<a_1<\ldots<a_n$ has all subset sums distinct, then
		\[a_n \geq \left(\sqrt{\frac{2}{\pi}}-o(1)\right)\cdot n^{-1/2} \cdot 2^n .\]
	\end{theorem}
	The second proof uses an isoperimetric inequality to show that $a_n \ge \binom{n}{\lfloor n/2 \rfloor}$ for all $n$, which asymptotically matches the lower bound above.
	The first proof uses the following result from probability.
	\begin{theorem}[Berry-Esseen]\label{BE} Let $X_1,\ldots,X_n$ be independent random variables with $\mathbb{E}[X_i]=0$, $\mathbb{E}[X_i^2]=\sigma_i^2$, and $\mathbb{E}[|X_i|^3] = \rho_i <\infty$. Let $X=X_1+\cdots+X_n$ and  $\sigma^2=\mathbb{E}[X^2]$. Then 
		$$\sup_{x \in \mathbb{R}} |F(x)-\Psi(x)| \leq C \cdot \psi,$$
		where $F(x)$ and $\Psi(x)$ are the cumulative distribution functions for $X$ and the normal distribution with mean zero and standard deviation $\sigma$, respectively, $C$ is an absolute constant, and $\psi=\sigma^{-3}\cdot \sum_{i=1}^n \rho_i$.
	\end{theorem}
	It is known  \cite{Shev} that one can take $C=0.56$ in the Berry-Esseen theorem. 
	
	Using the Berry-Esseen theorem and an inequality of Moser, our proof verifies Erd\H{o}s' conjecture that $a_n=\Omega(2^n)$ if the distribution of a random subset sum is not close to normal. If the distribution is close to normal, we 
	replace Chebyshev's inequality in the second moment argument by the normal distribution in bounding the probability that a random subset sum is in an interval. 
	
	\newpage
	
	\begin{proof}[Proof of Theorem~\ref{main}]
		Let $\epsilon_1,\ldots,\epsilon_n \in \{-1,1\}$ be independently and uniformly distributed and let $X=\epsilon_1a_1+\ldots+\epsilon_na_n$. 
		The random variable $X$ has mean $0$ and is symmetric around $0$. Moreover, each of its possible values occurs with probability $2^{-n}$ by the distinct subset sums property, and each of its possible values are of the same parity. Note that the variance of $X$ is $\sigma^2=\sum_{i=1}^n a_i^2$. 
		
		Let $\delta=\delta_n$ tend to $0$ slowly as $n$ tends to infinity, with $ \delta > n^{-1/2} $, e.g., $\delta = 1/\log n$ works.  We first notice that if  $\sigma \leq a_n/\delta$, then  Moser's lower bound on the variance \cite{Guy}:
		\[\sigma^2=\sum_{1\le i \le n} a_i^{2} \ge 1^{2} + 2^{2} + 4^{2} +\cdots 2^{2(n-1)} =\frac{4^{n}-1}{3}\]
		gives the desired lower bound $a_n \geq \delta \cdot 2^{n-1}>n^{-1/2}\cdot 2^{n}$.
		
		It remains to consider the case $\sigma > a_n/\delta$. We next apply Theorem~\ref{BE}. Note that $X=\sum_{i=1}^n X_i$, where $X_i=\epsilon_i a_i$. We have $\rho_i=a_i^3$  and $\sum_{i=1}^n \rho_i \leq a_n \cdot \sum_{i=1}^n a_i^2=a_n\sigma^2$, which implies $\psi \leq a_n/\sigma < \delta $ by the definition of $\psi$ and the bound on $\sigma$. By Theorem~\ref{BE}, we have $\sup_{x \in \mathbb{R}} |F(x)-\Psi(x)| \leq \delta$, that is, $X$ is $\delta$-close to normal. We bound $\textrm{Pr}[|X| \leq \ell]$ in two different ways to get a lower bound for $\sigma$.  
		
		Since the $2^n$ possible outcomes for $X$ have the same parity and are distinct, we obtain 
		\begin{equation}\label{upper}
		\textrm{Pr}[|X| \leq \ell] \leq (\ell+1)\cdot 2^{-n}.
		\end{equation}
		On the other hand, let $\ell=\alpha \sigma$ with $\alpha = \alpha_n$ tending to $0$ as $n$ tends to infinity with $\delta = o(\alpha)$, e.g., $\alpha  = \sqrt{\delta} $ works. As $X$ is $\delta$-close to a normal distribution, we obtain that
		\begin{equation}\label{asym}
		\textrm{Pr}[|X| \leq \ell] = \frac{1}{\sigma \sqrt{2 \pi}} \int_{-\ell}^{\ell}\exp \big(-\frac{x^{2}}{2\sigma ^{2}}\big)dx +O(\delta) \sim \sqrt{\frac{2}{\pi}} \, \frac{\ell}{\sigma}.
		\end{equation}
		Comparing \eqref{upper} and \eqref{asym}, we obtain that $\sigma \geq\left(\sqrt{2/\pi}-o(1)\right) \cdot 2^n$. Comparing this with the upper bound $\sigma^2 =a_1^2+\cdots+a_n^2 \leq na_n^2$, we have $a_n \geq \left(\sqrt{2/\pi}-o(1)\right)\cdot n^{-1/2} \cdot 2^n$ as desired.  
	\end{proof}

	The second proof will use the following special case of Harper's vertex-isoperimetric inequality. We write $2^{[n]}$ for the set of subsets of $\{1,2,\ldots, n\}$, and for $\mathcal{F}\subseteq 2^{[n]}$ we define the vertex boundary $\partial \mathcal{F} = \{A\in \mathcal{F}^c : |A\triangle B|=1 \text{ for some } B\in \mathcal{F}\}$.
	
	\begin{theorem}[Harper]\label{Harper} If $\mathcal{F}\subset 2^{[n]}$ such that $|\mathcal{F}|= 2^{n-1}$, then $|\partial \mathcal{F}|\geq \binom{n}{\lfloor n/2 \rfloor}$.
	\end{theorem}
	
	See \cite{Harp} for a full statement of Harper's theorem as well as \cite{Kleit} and \cite{Imre} for a particularly nice proof.

	\begin{proof}[A second proof of Theorem~\ref{main}]
		Set $a=(a_1,\ldots, a_n)$ and note that $a \cdot \epsilon \neq 0$ for all $\epsilon \in \{-\frac{1}{2},\frac{1}{2}\}^n$. Let $\mathcal{F}$ be the $2^{n-1}$ points $\epsilon$ in $\{-\frac{1}{2},\frac{1}{2}\}^n$ such that $a \cdot \epsilon <0$. Identifying $\{-\frac{1}{2},\frac{1}{2}\}^n$ with $2^{[n]}$ in the natural way, we have $\eta \in \partial \mathcal{F}$ if $\eta \pm e_i \in \mathcal{F}$ for some standard basis vector $e_i$. Theorem~\ref{Harper} then gives $|\partial \mathcal{F}| \geq \binom{n}{\lfloor n/2 \rfloor}$, where each $\eta \in \partial \mathcal{F}$ must satisfy $0< a\cdot  \eta < a_n$. Therefore there are distinct $\epsilon, \eta$ $\in \partial \mathcal{F}$ such that $|a \cdot (\epsilon - \eta)| \leq a_n/\binom{n}{\lfloor n/2 \rfloor}$. However, $\epsilon -\eta \in \{-1,0,1\}^n$, so $|a \cdot (\epsilon - \eta)| =|\sum_{i\in I} a_i - \sum_{j\in J} a_j|$ for some $I,J \subseteq [n]$, and we get $|a \cdot (\epsilon - \eta)| \geq 1$ by the distinct subset sum property. Comparing the lower and upper bound on $|a \cdot (\epsilon - \eta)|$ gives $a_n \geq \binom{n}{\lfloor n/2 \rfloor } = \left(\sqrt{2/\pi} - o(1) \right)\cdot n^{-1/2} \cdot 2^n$.
	\end{proof}
	{\bf Acknowledgements.} We would like to thank Noga Alon, Matthew Kwan, Bhargav Narayanan, Joel Spencer and Yufei Zhao for helpful comments.

\end{document}